\title{Linear combinations of Rademacher random variables
}
\author{Harrie Hendriks, Martien C.A. van Zuijlen}
\date{}

\documentclass[12pt]{article}
\usepackage{graphicx}
\usepackage{amsmath}
\usepackage{amsfonts}
\usepackage{amssymb}

\newtheorem{theorem}{Theorem}[section]

\newtheorem{lemma}[theorem]{Lemma}

\newtheorem{remark}[theorem]{Remark}

\def\qed{\ \rule{0.5em}{0.5em}}
\newenvironment{proof}[1][Proof]{\noindent\textbf{#1}\quad}{\ \rule{0.5em}{0.5em}}
\def\text{\mbox}
\def\R{{\mathbb R}}

\def\P{{\mathbb P}}

\def\E{{\bf E}}

\def\e{{\rm e}}

\def\QQ{\mathcal Q}
\def\SS{\mathcal S}

\def\curlyle{\preccurlyeq}
\def\curlyge{\succcurlyeq}

\begin{document}

\maketitle

\begin{abstract}
For a fixed unit  vector $a=(a_1,a_2,\ldots,a_n)\in S^{n-1},$ 
we consider the $2^n$ sign vectors 
$\varepsilon=
(\varepsilon^1,\varepsilon^2,\ldots,\varepsilon^n)
\in \{+1,-1\}^n$ 
and the corresponding scalar products 
$\varepsilon\cdot a=\sum_{i=1}^n \varepsilon^ia_i$. 
In this paper we will solve for $n=1,2,\ldots,9$ 
an old conjecture stating that of the $2^n$ sums of the form 
$\sum\pm a_i$ it is impossible that there are more with 
$|\sum_{i=1}^n \pm a_i|>1$ than there  are with 
$|\sum_{i=1}^n \pm a_i|\leq1.$ 
Although the problem has been solved completely in case the 
$a_i$'s are equal, the more general problem with possible non-equal 
$a_i$'s remains open for values of $n\geq 10.$ 
The present method can also be used for $n\geq 10,$ 
but unfortunately the technical difficulties seem to grow
exponentially 
with $n$ and no "induction type of argument" 
has been found. 
The conjecture 
has 
an appealing 
reformulation in probability theory and in geometry. 
In probability theory the results lead to upper bounds which are much
 better than for instance Chebyshev-inequalities.
\end{abstract}



\section{Introduction and result}
For a given $n$ we consider the set $\SS=\{+1,-1\}^n$ of 
row vectors
$\varepsilon=(\varepsilon^1,\ldots,\varepsilon^n)\in\R^n$ with $\varepsilon^i\in\{+1,-1\}, i=1,\ldots,n$.
Given a column vector $a=(a_1,\ldots,a_n)'$
one may consider the product $\varepsilon a=\sum_{i=1}^n\varepsilon^ia_i\in\R$.
The Euclidean length of $a$ will be denoted by $\|a\|$.
The problem we address is explicitly mentioned in \cite{G}
and acknowledged to Bogus\l{}av Tomaszewki: Of the $2^n$ expressions
$|\varepsilon a|=|\varepsilon^1a_1+\cdots+\varepsilon^na_n|$, $\varepsilon\in\SS$,
 
\centerline {\bf Can there be more with value
$>\|a\|$ than with value $\le\|a\|$?}
\noindent
The result proved in this paper, is the solution for $n\le9$:
\begin{theorem} \label{Main}
Let $n\le9$ and $a\in\R^n$. 
Then 
$\{\varepsilon\in\SS\mid |\varepsilon a|\le\|a\|\}$ contains at least half of the elements of $\SS$.
\end{theorem}
%
Let for the moment $\epsilon_1,\epsilon_2,\ldots,\epsilon_n$ 
be independent Rademacher random variables and let 
$$ S_n:=a_1\epsilon_1+a_2\epsilon_2+\cdots+a_n\epsilon_n,$$
where $a_1,a_2,\ldots,a_n$ are any real numbers such that 
$a_1^2+a_2^2+\cdots+a_n^2=1.$ 
In \cite{P}, Pinelis showed that 
$$\P\{S_n\geq x\}\leq c\,\P\{Z\geq x\} $$ where $Z$ is the standard normal r.v. and $c$ is an almost optimal constant.
He indicated that, 
while $S_n$ represents a simplest case of 
the sum of independent non-identically distributed r.v.'s, 
it is still very difficult to control in a precise manner. 
As a by-product he obtained
$$\P\{S_n\geq x\}\leq h_1(x),$$
%
where 
\begin{align*}
h_1(x)=
\begin{cases}
1, &\text{ for }  x\le0,\\
\frac12,  &\text{ for }  0<x\le1,\\
\frac1{2x^2},  &\text{ for }  1\le x<\sqrt2,\\
g(x),  &\text{ for } \sqrt2\le x<\sqrt3,\\
h(x),  &\text{ for }  x\ge\sqrt3,\\
\end{cases}
\end{align*}

\noindent 
and where $g$ and $h$ are complicated expressions involving the
standard normal density and the standard normal distribution function, 
but $g$ and $h$ are independent of $n$ and of the $a_i$'s.
In \cite{P1}, Pinelis  proved that for $x>0$
$$\P\{S_n\geq x \}
\leq 
\P\{Z> x \}+\frac{C\phi(x)}{9+x^2}<\P\{Z> x \}(1+\frac{C}{x}) ,$$
where $C:=5\sqrt{2\pi\e}\,\P\{|Z|< 1 \}=14.10\ldots,$ 
and where $\phi$ is the standard normal density function.
Holzman and Kleitman showed in \cite{HK} that
$$\P\{|S_n|< 1\}\geq \frac{3}{8},\hbox{ if }|a_i|<1,\hbox{ for }i=1,\ldots,n,$$
so that also
$$\P\{|S_n|\leq  1\}\geq \frac{3}{8},$$
but the conjecture that
$$\P\{|S_n|\leq  1\}\geq \frac{1}{2}$$
remained open. 
In \cite{MVZ}, however, Van Zuijlen proved this conjecture in case of equal $a_i$'s.
Bentkus and Dzindzalieta obtained in \cite{BD} in case  
$a_1^2+a_2^2+\cdots+a_n^2\leq 1$ the inequality
$$\P\{S_n\geq x\}\leq c\P\{Z\geq x\},\text{ for all }x\in \R,$$
where $c=(4\P\{Z\geq \sqrt{2}\})^{-1}=3.178\ldots$ is optimal. 
Bentkus and Dzindzalieta obtained in \cite{BD} the inequality
$$\P\{S_n\geq x\}\leq \frac{1}{4}+\frac{1}{8}(1-\sqrt{2-2/{x^2}}),\text { for } x\in(1,\sqrt{2}].$$
It has also been indicated in \cite[p. 31]{D} how to use Lyapunov type bounds, with explicit constants for the remainder term in the Central Limit Theorem, in order to obtain upper bounds for $\P\{S_n\geq x\}.$ 
Let $Y_1,Y_2,\ldots,Y_n$ be independent rv's such that $\E Y_j=0$ for all $j$. 
Denote $\beta_j=\E|Y_j|^3.$ 
Assume that the sum $U_n=Y_1+Y_2+\cdots+Y_n$ has unit variance. 
Then there exits an absolute constant, say $c_L,$ such that
$$|\P\{U_n\geq x\}-(1-\Phi(x)|\leq c_L(\beta_1+\beta_2+..+\beta_n). $$ 
It is known that $c_L\leq 0.56.$ 
By replacing $Y_j$ by $a_j\epsilon_j$ and using $\beta_j=|a_j|^3\leq \tau a_j^2$
with $\tau=\max\{|a_j|, j=1,\ldots,n\}$ we obtain
$$|\P\{S_n\geq x\}-(1-\Phi(x)|\leq c_L\tau. $$
Dzindzalieta, Ju\v{s}kevi\v{c}ius and \v{S}ileikis proved in \cite{DJS} that, for $M_n=X_1+X_2+\cdots+X_n$ a sum of independent symmetric rv's such that $|X_i|\leq 1,$ 
$$\P\{M_n\in A\}\leq\P\{cW_k\in A\},$$ 
where $A$ is either an interval of the form $[x,\infty)$ or just a single point
and $W_k=\epsilon_1+\cdots+\epsilon_k$.
The optimal values $c$ and $k$ are given explicitly. It improves Kwapie\'n's inequality in the case of Rademacher series.
\\
Oleszkiewicz proved in \cite{O} for $n>1$ and any real numbers $a_1\ge a_2\ge\cdots\ge a_n>0$, that
$$\P\left\{S_n> \sqrt{\sum_{i=1}^n a_i^2}\right\}> \frac{1}{20},$$
and for $n=6$ even that, for any $a_1,\ldots,a_6$,  
$$\P\left\{S_6\geq \sqrt{\sum_{i=1}^6 a_i^2}\right\}\geq \frac{7}{64}.$$
In \cite{HvZ} Hendriks and Van Zuijlen studied, in the case 
$a_1=a_2=\ldots=a_n=n^{-1/2}$, the probabilities $\P\{|S_n|\le\xi\}$ for $\xi\in(0,1]$, depending on $n$.
Furthermore, we have to mention the interesting papers \cite{vH} and \cite{Z} of 
Von Heymann and of Zhubr, where the problem has been studied from a geometrical point of view.
\\
In Section \ref{Methods} we explain how the proof is devised.
In Section \ref{Proofsnew} we exhibit the ingredients of the proof based on Table \ref{Proof scheme new experimental}.
At the end of this section we present a slight variation on the result of Holzman and Kleitman in \cite{HK}.
In the Appendix we collect some ideas about the structure of 
the problem, and we give some background on quadratic programming.

\section{Methods}\label{Methods}
For vectors $x,y\in\R^m$ we will denote by $x\curlyge y$
the property that $x-y\in\R_+^m$, where $\R_+=[0,\infty)$. 
For $x\in\R^m$ and $y\in\R$ we will denote by $x\curlyge y$
the property that $x\curlyge\underline1\,y$ where $\underline1\in\R^m$ denotes the
vector with all coefficients equal to 1.
\\
In considering the number of elements of 
$\{\varepsilon\in\SS\mid |\varepsilon a|\le\|a\|\}$
it is clear that without loss of generality we may 
assume that 
$a\in\QQ:=\{x=(x_1,\ldots,x_n)'\mid x_1\ge x_2\ge\cdots\ge x_n\ge0\}$.
Let $\SS^+$ be the subset of $\SS$ of sign vectors 
$\varepsilon=(\varepsilon^1,\ldots,\varepsilon^n)$ with 
$\varepsilon^1=+1$.
It is clear that $\{\varepsilon\in\SS^+\mid |\varepsilon a|\le\|a\|\}$
(resp. $\SS^+$)
contains half of the elements of 
$\{\varepsilon\in\SS\mid |\varepsilon a|\le\|a\|\}$ (resp. $\SS$) 
so it is sufficient to prove that for all $a\in\QQ$
at least half of the elements $\varepsilon\in\SS^+$ satisfy 
$|\varepsilon a|\le\|a\|$.
\\
We introduce the \emph{dual} $\QQ^*$
of $\QQ$ to be defined as the set of row vectors 
$V$ such that $Va\ge0$ for all $a\in\QQ$.
Since $\QQ$ is the set of linear combinations with non-negative coefficients of the $n$ elements
(counting the number of ones) of the form $(1,\ldots,1,0,\ldots,0)$,
for $V$ to belong to $\QQ^*$ it is necessary and sufficient that the vector of cumulative sums of $V$,
$\overline V:=(V^1,V^1+V^2,\ldots,V^1+V^2+\cdots+V^n)$, satisfies $\overline V\curlyge0$.
%
%
\\
Basically the proof of Theorem \ref{Main} is a compilation of cases where we use the following
argument:
\begin{lemma}\label{lemma 2}
  Let $e_1,\ldots,e_k\in\SS$ for some $k\ge1$ and $\sigma^1,\ldots,\sigma^k\in\{+1,-1\}$.
  Suppose there exist row vectors $R\in\R^n$ and $\lambda=(\lambda^1,\ldots,\lambda^k)\curlyge0$  
  with $\|R\|\le1$ and
  $\lambda^1+\cdots+\lambda^k=1$,
 such that
   $R-L\in Q^*$ 
  (i.e. $La\le Ra$ for all $a\in\QQ$), where
  $L=\sum \lambda^\ell\sigma^\ell e_\ell$.  
  Then for all $a\in\QQ$ we have the implications
  $$\left[ \forall\ell:\sigma^\ell e_\ell a\ge0 \right]
    \Rightarrow\left[ \min_\ell |e_\ell a|\le La\le Ra\le \|R\|\,\|a\|\le\|a\| \right]
	\Rightarrow\left[ \exists\ell:|e_\ell a|\le\|a\| \right].\qed
$$
\end{lemma}  
Given the signs $\sigma^1,\ldots,\sigma^k$, there are
many combinations $e_1,\ldots,e_k\in\SS$ that share the same 
vectors $\lambda$ and $R$, although, generally, 
the vector $L$ depends
on the particular combination. 
Suppose $R$ and $\lambda$ as in Lemma \ref{lemma 2} are found, 
so that in particular $\|R\|\le1$.
Then, without changing its norm, by taking absolute values and sorting
the coefficients of vector $R$, one can arrange that $R'\in\QQ$ and still 
satisfies the conditions of Lemma \ref{lemma 2}.
In the Appendix it is shown, that $R$ and $\lambda$ can be found as in
Lemma \ref{lemma 2}, such that $R$ itself satisfies $R'\in\QQ$ and
$ \sigma^\ell e_\ell R'\ge0,\ \ell=1,\ldots,k$,
but such restriction on $R$ would lead to unnecessary complications in our arguments.
\\
We enumerate the elements of $\SS$ as follows:
$$\varepsilon=(\varepsilon^1,\ldots,\varepsilon^n)=\varepsilon_i
\text{ with }
i=\sum_{j=0}^{n-1}\frac{1-\varepsilon^{n-j}}22^j.$$
Then $\SS^+$ is the subset of $\SS$ of the $\varepsilon_i$ with $0\le i<2^{n-1}$.
Often we will identify $\SS^+$ with the numerals $\{i\mid 0\le i<2^{n-1}\}$.
In general we will call an ordered set $(i_1,\ldots,i_k)\in(\SS^+)^k$ a \emph{$k$-tuple}
(or \emph{pair}, for $k=2$) and call it a \emph{$k$-tuplet} (or \emph{twin}, for $k=2$), 
if for any $a\in\QQ$
$\#\{\ell\mid|\varepsilon_{i_\ell}|\le\|a\|\}\ge k/2$.
 A pair $(i,j)\in\SS^+\times\SS^+$ will be called \emph{conjugate}
if $i+j=2^{n-1}-1$ or, equivalently
$\frac12\varepsilon_i+\frac12\varepsilon_j=(1,0,\ldots,0)$.
We will refer to $i$ and $j$ as the legs of the pair.
\\
\begin{proof}[Proof of Theorem \ref{Main}]
It is rather obvious to see, that we need only consider the case
$n=9$.
The proof follows the scheme given in Table \ref{Proof scheme new experimental}.
The table gives a partition of the sign vectors in $\SS^+$ in unions of conjugate pairs enclosed
in outer parentheses, and the conjugate twins that are not mentioned in the table.
In the next section we will prove that the thus given $2k$-tuples form $2k$-tuplets.
\end{proof}

\section{Proofs}
\label{Proofsnew}
We base the proof of Theorem \ref{Main} on the proof scheme in Table \ref{Proof scheme new experimental}. 
Recall that in proving Theorem \ref{Main} it is enough to show
that for all $a\in\QQ$, for at least half of the sign-vectors
$\varepsilon\in\SS^+$ 
it holds that $|\varepsilon a|\le\|a\|$.

\begin{table}[hbt]
\setlength{\tabcolsep}{2.5pt}

\begin{tabular}{cccccccccc||ccccccc}
(0&255&(94&161)&(105&150)&(109&146))  &&&& (32&223&(102&153)) \\
(1&254&(95&160)&(104&151)&(108&147))  &&&& (33&222&(103&152)) \\
(2&253&(93&162)&(106&149))  &&&&&& (34&221&(101&154)) \\
(3&252&(92&163)&(107&148))&&&&&& (35&220&(100&155)) \\
(4&251&(91&164)&(99&156))&&&&&&\\
(5&250&(90&165)) &&&&&&&& (64&191&(54&201))\\
(6&249&(89&166)) &&&&&&&& (65&190&(55&200))\\
(7&248&(88&167)) &&&&&&&& (66&189&(53&202))\\
(8&247&(87&168)&(113&142)) &&&&&& (67&188&(52&203)) \\
(9&246&(86&169)) &&&&&&\\
(10&245&(85&170)) &&&&&&&& (128&127&(58&197))  \\
(11&244&(84&171)) &&&&&&&& (129&126&(59&196))  \\
(12&243&(83&172)) &&&&&&&&  (130&125&(57&198)) \\
(13&242&(82&173))&&&&&&&& (131&124&(56&199))\\
(14&241&(81&174))\\
(15&240&(80&175))\\
(16&239&(79&176)&(114&141))\\
(17&238&(78&177)) \\
(18&237&(77&178))\\
(19&236&(76&179))\\
(20&235&(75&180))\\
&&& \\
(24&231&(71&184)) 
\end{tabular}
\caption{Proof scheme for $n=9$, numbers should not appear twice.
Terms enclosed in matching parentheses correspond to
twins, 4-tuples, 6-tuples, 8-tuples.
Non mentioned numbers belong to conjugate twins.
}\label{Proof scheme new experimental}
\end{table}

\subsection*{Conjugate twins}
A pair $(i,j)$ or $(\varepsilon_i,\varepsilon_j)\in\SS^+\times\SS^+$ 
with $i\ne j$ 
is a \emph{twin} if for all $a\in\QQ$
it holds that $\min(|\varepsilon_ia|,|\varepsilon_ja|)\le\|a\|$.
Of course, if we could partition $\SS^+$ into disjoint twins, we would have proved
Theorem \ref{Main}.
Unfortunately, such partitioning is not possible.
\\
Let $R_1=(1,0,\ldots,0)\in\R^n$ and $R_0=(0,\ldots,0)\in\R^n$.
Conjugate pairs $(i,j)$ are rather special in the following sense:
if $a\in\QQ$, then $\frac12\varepsilon_ia+\frac12\varepsilon_ja=R_1a=a_1\ge0$.

\begin{remark}\label{++new}
Let $(\varepsilon_i,\varepsilon_j)$ be a conjugate pair and $a\in\QQ$. 
Then
$$R_1-\frac12\varepsilon_i-\frac12\varepsilon_j=0\in\QQ^*
\hbox{ and }R_0+\frac12\varepsilon_i+\frac12\varepsilon_j=R_1\in\QQ^*.$$
Then
\begin{itemize}
\item
If $\min(\varepsilon_ia,\varepsilon_ja)\ge0$, then
$\min(|\varepsilon_ia|,|\varepsilon_ja|)\le|R_1a|\le\|a\|$
 (Lemma \ref{lemma 2}, with $k=2$ and $\sigma^1=\sigma^2=1$)
\item
If both 
$\varepsilon_ia\le0$ and $\varepsilon_ja\le0$ then $\min(|\varepsilon_ia|,|\varepsilon_ja|)\le R_0a=0\le\|a\|$  (Lemma \ref{lemma 2}, with $k=2$ and $\sigma^1=\sigma^2=-1$).
In fact, then $a_1=R_1a\le0$ so that $a=0$.
\end{itemize}
Notice that $|\min(\varepsilon_ia,\varepsilon_ja)|=\min(|\varepsilon_ia|,|\varepsilon_ja|)$.
\\
Thus  $(\varepsilon_i,\varepsilon_j)$ is a twin if and only if
$\left[\min(\varepsilon_ia,\varepsilon_ja)<0\right]
\Rightarrow \left[\,|\min(\varepsilon_ia,\varepsilon_ja)|\le\|a\|\right]$.
\qed
\end{remark}
From now on we assume $n=9$.
Let $R_2=\frac12(1,1,1,1,0,\ldots,0)\in\R^9$ and 
$R_3=\frac13(1,\ldots,1)\in\R^9$.

\begin{theorem}\label{Twinsnew} 
Let $n=9$.
Of all the 128 conjugate pairs 
$(\varepsilon_i,\varepsilon_j)\in\SS^+\times\SS^+$ with $i<j,$ exactly 34 of them  are non-twin,
namely  the pairs
$$(i,j)=(255-j,j),\text{ with } j=128, \ldots,131,188,\ldots,191,220,\ldots,223,231,235,\ldots,255.$$
\end{theorem}
\goodbreak
\begin{proof}
Given $\varepsilon_\ell$, we show that there exists a vector $R$ such that $R+\varepsilon_\ell\in\QQ^*$ and $\|R\|\le1$
so that Lemma \ref{lemma 2} ($k=1$, $\sigma^1=-1$) applies,
or there is a vector $R$ such that $R'\in\QQ$, $\|R\|>1$ and $-\varepsilon_\ell R'\ge RR'$.
These vectors $R$ are found as feasible solution if not the optimal solution of a quadratic programming problem as explained in Theorem \ref{SQP} and Remarks \ref{Remark5} and 
\ref{Remark6}.
\begin{itemize}
\item
The sign vectors $\varepsilon_\ell$ with $0\le \ell\le 123$,  $128\le \ell\le 187$ or $192\le \ell\le219$ satisfy $R_3+\varepsilon_\ell\in\QQ^*$.
In particular
$0\le-\varepsilon_\ell a\le R_3a\le\|R_3\|\|a\|=\|a\|$ for all $a$ satisfying 
$0\le-\varepsilon_\ell a$.
Notice that the given sign vectors 
$\varepsilon_\ell$ are exactly those satisfying
$\varepsilon_\ell-\varepsilon_{219}\in\QQ^*$.
\item
For $R=\frac37(1,1,1,1,1,1,1,0,0)$ and 
$124\le \ell\le 127$, $188\le \ell\le 191$ and $220\le \ell\le 223$ 
we have $\|R\|^2=RR'=\frac97$ and $\varepsilon_\ell R'=-\frac97<0$, 
so that $|\varepsilon_\ell R'|>\|R\|$.
\item
For $224\le \ell\le 230$ and $232\le \ell\le234$ we have $R_2+\varepsilon_\ell\in\QQ^*$.
Notice that the sign vectors $\varepsilon_\ell$  with $R_2+\varepsilon_\ell\in\QQ^*$
are exactly those satisfying
$
\varepsilon_\ell -\varepsilon_{234}\in\QQ^*$.
\item
For $\ell=231$ and $\ell=235$ and $R=\frac1{10}(5,5,5,5,2,2,2,2,2)$ we have $RR'=\frac65$ and $\varepsilon_\ell R'=-\frac65<0$, so that $|\varepsilon_\ell R'|>\|R\|$.
\item
For $236\le \ell\le 239$ and $R=\frac37(1,1,1,1,1,1,1,0,0)$ we have $RR'=\frac97$
and $\varepsilon_\ell R'=-\frac97<0$, so that $|\varepsilon_\ell R'|>\|R\|$.
\item
For $240\le \ell\le 255$ and $R=\frac35(1,1,1,1,1,0,0,0,0)$ we have $RR'=\frac95$
and $\varepsilon_\ell R'=-\frac95<0$, so that $|\varepsilon_\ell R'|>\|R\|$.
\end{itemize}
The non-twins are those conjugate pairs $(\varepsilon_i,\varepsilon_j)$ 
where $\left[\varepsilon_\ell a<0\right]\Rightarrow\left[\,|\varepsilon_\ell a|\le\|a\|\right]$ fails for 
$\ell=i$ or for $\ell=j$.
\end{proof}

\begin{remark}
Notice two curious facts
\begin{itemize}
\item
$\varepsilon_j\in \QQ^*$ for $j=128, \ldots,131$ 
and 
\\
$\varepsilon_{255-j}\in \QQ^*$ for $j= 188 \ldots,191,220,\ldots,223,231,235,\ldots,255.$ 
\\
In particular, for each non-twin conjugate pair, one of the two legs belongs to $\QQ^*.$
\item
Among the 94 conjugate twins there are 40 special twins in the sense that one of the legs $\varepsilon$ in these twins satisfy
$|\varepsilon a|\le\|a\|$ for all $a\in Q.$ 
These are the $\varepsilon_j$,  with
$j=170,\ldots,175,178,\ldots,187,202,\ldots,207,210,\ldots,219,226,\ldots,230,232,233,234$.
\goodbreak\noindent
For these $\varepsilon_j$ we have $R_1-\varepsilon_j\in\QQ^*$ and 
[$R_2+\varepsilon_j\in\QQ^*$ or $R_3+\varepsilon_j\in\QQ^*$].
\qed
\end{itemize}
\end{remark}

\subsection*{Proof witnesses}
We will be able to combine any conjugate pair $(e_1,e_2)$ with $k-1$ ($k=1,2,3,4$) conjugate twins  $(e_{2l-1},e_{2l})$, $l=2,...,k$, yielding a $2k$-tuple
$A:=(e_1,e_2,(e_3,e_4),(e_5,e_6),...,(e_{2k-1},e_{2k}))$,
with the property that for all $a$ in $\QQ$, at least $k$ terms of $(|e_{1}a|,|e_{2}a|,...,|e_{2k-1}a|,|e_{2k}a|)$ are not greater than $\|a\|$, proving that the $2k$-tuple is a $2k$-tuplet.
\\
Let $a\in\QQ$, $a\ne0$. Say $\varepsilon\in\SS$ is \emph{good} if $|\varepsilon a|\le\|a\|$.
In order to show that $A$ is a $2k$-tuplet, we notice that already each of the $k-1$ twins has a good leg, we therefore leave out from $A$
from each twin $(e_{2l-1},e_{2l})$ that sign vector $e$ for which $|ea|=\min(|e_{2l-1}a|,|e_{2l}a|)$. For this $e$ we have $|ea|\le \|a\|$ and for the remaining sign vector $e^*$ we necessarily have $e^*a\ge 0$. 
So in the remaining $k+1$-tuple $\tilde A$ we need only find one good sign vector.
\\
For the  conjugate pair $(e_1,e_2)$ we recall that if $e_1a\ge0$ and $e_2a\ge0$ 
then at least one of 
$e_1a,e_2a$ is not greater than $\|a\|$ and 
we have found a good sign vector. 
If not, say $e_1a<0$, then we have $0<-e_1a\le e_2a$, and if $\min(|e_1a|,|e_2a|)\le||a||$, then $0<-e_1a\le\|a\|$ and we can also leave out $e_2$ from $\tilde A$, so that in the remaining 
$k$-tuple $A^*$ we still have to find a good sign vector.
\\
Summarizing it is sufficient to prove that for all combinations of $k$ positions
$s_1,\ldots,s_k$ with $s_1\in\{1,2\},\ldots,s_k\in\{2k-1,2k\}$
and signs $\sigma^1=-1,\sigma^2=+1,\ldots,\sigma^k=+1$, condition 
$\forall i:\sigma^ie_{s_i}a\ge0$ implies that there is some $i$ with $|e_{s_i}a|\le\|a\|$.
\\
For any such combination, by Lemma \ref{lemma 2}, it is sufficient to produce a row-vector $R\in\R^9$ and
a vector $\lambda\in\R^k$, such that $\|R\|\le1$, $\lambda\curlyge0$, 
$\lambda^1+\cdots+\lambda^k=1$ and
$R-\sum\lambda^i\sigma^ie_{s_i}\in\QQ^*$.
We will use the following notation
\begin{center}
\begin{tabular}{ccccccccccccccccccccccc}
$(2k)$-tuplet&case&R&$\lambda$\\
$(e_1,e_2,(e_3,e_4),\ldots,(e_{2k-1},e_{2k}))$&$[s_1,\ldots,s_k]$ & $R=(R^1,\ldots,R^9)$ & $(\lambda^1,\ldots,\lambda^k)$
\end{tabular}
\end{center}
to claim that $R-L\in\QQ^*$ where 
$L=\sum\lambda^i\sigma^i \varepsilon_{s_i}$
with $\sigma^1=-1,~\sigma^2=\ldots=\sigma^k=+1$.
%
Clearly if some $\lambda^\ell=0$, the sign vector $e_{s_\ell}$ is  
irrelevant, which we will denote by putting a '$*$' instead of $s_\ell$. 

\subsection*{2-tuplets (conjugate twins revisited)}
In the notation proposed above, we can give witnesses of the proof for 
conjugate twins (i.e. $k=1$) 
as follows.
Recall the understood sign is $\sigma^1=-1$.
\\
\\
\begin{tabular}{ccccccccccccccccccccccccccc}
twin&case&R&$\lambda$\\
(21,234)&[1]&$R_3=(1,1,1,1,1,1,1,1,1)/3$&1\\
&[2]&$R_2=(1,1,1,1,0,0,0,0,0)/2$&1\\
\end{tabular}
\\
These values hold for all twins $(i,255-i)$ with
$i=21,22,23,25,\ldots,31,37,\ldots,39,41,\ldots,47,\\49,\ldots,63,69,\ldots,71,
73,\ldots,79,81,\ldots,95,97,\ldots,123$.
It also works for the remaining twins with $i=255-j$ with $j=36,40,48,68,72,80,96$.
\\
\\
Recall the meaning of the above table, for example for $i=31$, 
is as follows
\begin{align*}
R_3-\lambda\sigma^1\varepsilon_{31}
&=
(1,1,1,1,1,1,1,1,1)/3+(1,1,1,1,-1,-1,-1,-1,-1)
\\&=
(4,4,4,4,-2,-2,-2,-2,-2)/3\in\QQ^*
\\
R_2-\lambda\sigma^1\varepsilon_{224}
&=
(1,1,1,1,0,0,0,0,0)/2+(1,-1,-1,-1,1,1,1,1,1)
\\&=
(3,-1,-1,-1,2,2,2,2,2)/2\in\QQ^*
\end{align*}

\begin{remark}\label{semi-8-tuples}
Fortunately a proof can be constructed around the idea to 
combine any conjugate pair with conjugate twins in order to find subsets
of $\SS^+$ that form $(2k)$-tuplets.
For (very) high $n$ this is certainly not possible any more, since the
fraction of conjugate twins among conjugate pairs asymptotically converges to 0.
For $n=9$, there happen to be 8-tuples comprising 4 conjugate pairs of which 
only one is conjugate twin. Examples are
\\
(7, 248, 20, 235, 33, 222, (77, 178)) and (15, 240,  24, 231,  66, 189,  (87, 168)).
\end{remark}
%
%

\subsection*{4-tuplets}
We will first prove that the 15 combinations $(i,255-i,(95-i,160+i))$
with 
\\
$i=5,6,7,9,\ldots,15,17,\ldots,20,24$ are 4-tuplets.
\\
\begin{tabular}{ccccccccccccccccccccccccccc}
4-tuplet&case&R&$\lambda$\\
(5,250,(90,165))&[1, *]&$(0,0,0,0,0,0,0,0,0)$&$(1,0)$\\
&[2, 3]&$R_2=(1,1,1,1,0,0,0,0,0)/2$&$(1,1)/2$\\
&[2, 4]&$R_3^*=(2,1,1,1,1,1,0,0,0)/3$&$(1,5)/6$\\
\end{tabular}
\\
\\
These values hold for all 15 4-tuples $(255-j,j,j-160,255-j+160)$ with 
$j=231,235,\ldots,238,\\240,\ldots,246,248,\ldots,250$.
More precisely, for those $j$: 
$L=\frac12(-\varepsilon_j)+\frac12\varepsilon_{j-160}=\frac12(-\varepsilon_{250})+\frac12\varepsilon_{90}=(0,1,0,1,0,0,0,0,0).$
We then have
\begin{align*}
R_2-L&=(1,1,1,1,0,0,0,0,0)/2-(0,1,0,1,0,0,0,0,0)
\\&=(1,-1,1,-1,0,0,0,0,0)/2\in\QQ^*
\end{align*}
Furthermore, for those $j$, 
with $L_5=\frac16(-\varepsilon_{250})+\frac56\varepsilon_{165}$ we have:
$L_5-
(\frac16(-\varepsilon_j)+\frac56\varepsilon_{255-j+160})\in\QQ^*$
and
\begin{align*}
R_3^*-L_5&=(2,1,1,1,1,1,0,0,0)/3-(2,-2,3,-2,3,3,-3,-3,3)/3
\\&=(0,3,-2,3,-2,-2,3,-3,3)/3\in\QQ^*.
\end{align*}
\\
There are 12 more 4-tuplets.
\\
\begin{tabular}{ccccccccccccccccccccccccccc}
4-tuplet&case&R&$\lambda$\\
(32,223,(102,153))
&[2, *]&$(0,0,0,0,0,0,0,0,0)$&$(1,0)$\\
&[2, 3]&$R_5=(3,3,1,1,1,1,1,1,1)/5$&$(2,3)/5$\\
&[2, 4]&$R_3=(1,1,1,1,1,1,1,1,1)/3$&$(1,2)/3$\\
\end{tabular}
\\
These values also hold for (33,222,(103,152)), (34,221,(101,154)), (35,220,(100,155)),
\goodbreak\noindent 
(64,191,(54,201)); (65,190,(55,200)); (66,189,(53,202)); (67,188,(52,203));
(128,127,(58,197)); (129,126,(59,196)); (130,125,(57,198)); (131,124,(56,199)).

\subsection*{6-tuplets}
We need consider the following five cases\\
(2,253,(93,162),(106,149));
(3,252,(92,163),(107,148));
(4,251,(91,164),(99,156));\\
(8,247,(87,168),(113,142)); 
(16,239,(79,176),(114,141)).
\\

\begin{tabular}{ccccccccccccccccccccccccccc}
6-tuplet
&case&R&$\lambda$\\
(2,253,(93,162),(106,149))
&[1, *, *]&$(0,0,0,0,0,0,0,0,0)$&$(1,0,0)$\\
&[2, 3, *]&$R_2=(1,1,1,1,0,0,0,0,0)/2$&$(1,1,0)/2$\\
&[2, 4, 5]&$R_3^*=(4,2,2,2,2,1,1,1,1)/6$&$(1,2,3)/6$\\
&[2, 4, 6]&$R_3^*=(2,1,1,1,1,1,0,0,0)/3$&$(1,2,3)/6$
\end{tabular}
\\
the same values hold for (3,252,(92,163),(107,148)).
\\
\\
\begin{tabular}{ccccccccccccccccccccccccccc}
6-tuplet
&case&R&$\lambda$\\
(4,251,(91,164),(99,156))
&[1, *, *]&$(0,0,0,0,0,0,0,0,0)$&$(1,0,0)$\\
&[2, 3, *]&$R_2=(1,1,1,1,0,0,0,0,0)/2$&$(1,1,0)/2$\\
&[2, 4, 5]&$R_3^*=(2,1,1,1,1,1,0,0,0)/3$&$(1,2,3)/6$\\
&[2, 4, 6]&$R_6=(4,2,2,2,2,1,1,1,1)/6$&$(1,3,2)/6$\\
\end{tabular}
\\
the same values also hold for (8,247,(87,168),(113,142))
and (16,239,(79,176),(114,141)).

\subsection*{8-tuplets}
We need consider the two cases
(0,255,(94,161),(105,150),(109,146))
and\\
(1,254,(95,160),(104,151),(108,147)).
\\
\\
\begin{tabular}{ccccccccccccccccccccccccccc}
\multispan{2}8-tuplet (0,255,(94,161),(105,150),(109,146))\\
case&R&$\lambda$\\
{} [1, *, *, *]&$(0,0,0,0,0,0,0,0,0)$&$(1,0,0,0)$\\
{} [2, *, *, 7]&$R_5=(3,3,1,1,1,1,1,1,1)/5$&$(2,0,0,3)/5$\\
{} [2, 3, *, 8]&$R_2=(1,1,1,1,0,0,0,0,0)/2$&$(1,2,0,1)/4$\\
{} [2, *, 5, 8]&$R_4=(3,1,1,1,1,1,1,1,0)/4$&$(1,0,4,3)/8$\\
{} [2, 4, 6, *]&$R_3^*=(2,1,1,1,1,1,0,0,0)/3$&$(1,1,4,0)/6$
\end{tabular}
\\
Same works for the other 8-tuplet.
\\
\\
Notice that the conjugate pair $(1,254)$ also fits in a 6-tuplet, e.g. (1,254,(94,161),(86,169)).
We were not able to exploit this fact since the involved twins are
needed elsewhere.
%
%
%
\\
\\
We have collected in Table \ref{library} the vectors $R\in\R^9$ needed in the proof
of Theorem \ref{Main}.
\begin{table}[h]
\begin{align*}
R_0&=(0,0,0,0,0,0,0,0,0);\\
R_1&=(1,0,0,0,0,0,0,0,0);\\
R_2&=(1,1,1,1,0,0,0,0,0)/2;\\
R_3&=(1,1,1,1,1,1,1,1,1)/3;\\
R_3^*&=(2,1,1,1,1,1,0,0,0)/3;\\
R_4&=(3,1,1,1,1,1,1,1,0)/4;\\
R_5&=(3,3,1,1,1,1,1,1,1)/5;\\
R_6&=(4,2,2,2,2,1,1,1,1)/6;
\end{align*}
\caption{Table of vectors used in the proof}\label{library}
\end{table}

\subsection*{An additional result}
In relation to the results in Holzman and Kleitman \cite{HK}
consider the maximal numbers $c_k$, $k=1,\ldots,n$, such that
$$\frac1{2^n}\#\{\varepsilon\in\SS\mid|\varepsilon a|<\|a\|\}\ge c_k,\hbox{ if }|a_i|>0, \hbox{ for }i\le k,\hbox{ and }a_i=0 \hbox{ for }i>k.$$
The result in \cite{HK} implies that $c_k\ge3/8$ for $k\ge2$.
Notice that for any $R=(r_1,r_2,0,\ldots,0)$ with $r_1\ge r_2>0$,
it holds that $|\varepsilon R|\in\{r_1-r_2,r_1+r_2\}$.
Of course $0\le r_1-r_2\le\|R\|=\sqrt{r_1^2+r_2^2}<r_1+r_2$, and even
$2^{-n}\#\{\varepsilon\in\SS\mid|\varepsilon R|<\|R\|\}
=2^{-n}\#\{\varepsilon\in\SS\mid \varepsilon^1=-\varepsilon^2\}
=\frac12$.
This result also holds for vectors $a$ sufficiently close to $R$, say for $a$
satisfying $\|R-a\|<\delta$ where $\delta$ depends on $R$.
Of course there are vectors $a\in\QQ$ with $a_i>0$, for $i=1,\ldots,n$ and $\|R-a\|<\delta$. 
Therefore, $c_k\le1/2$ for all $2\le k\le n$.
\begin{theorem}\label{HKext}
Let $n=9$.
We have $c_1=0$ attained for $a=R_1$, $c_2=c_3=1/2$, $c_4=3/8$ attained for $a=R_2$, $c_5=1/2$, 
$c_6=15/32$ attained for $a=R_3^*$, $c_7=1/2$,
$c_8=7/16$ attained for $a=R_4$ and $c_9=63/128$ attained for $a=R_3$.
\end{theorem}
\begin{proof}
For any $a\in\QQ$ 
we have shown that for at least half of $\varepsilon\in\SS^+$ there is $R$  from the above table,
$R$ depending on $\varepsilon$, such that $|\varepsilon a|\le Ra$.
Moreover, 
if $a\in\QQ$ is not a scalar multiple of any $R$ from this table, the inequality
$Ra\le\|a\|$ is strict.
So we need only consider the vectors $a$ from Table \ref{library}.
The theorem follows simply by counting the number of sign sequences $\varepsilon$ with $|\varepsilon R|<\|R\|$.
\end{proof}

\section{Appendix}
\subsection*{Group structure}
Notice that there is a product structure given by $\varepsilon_i*\varepsilon_j=\varepsilon_k$
where $\epsilon_k^\ell=\epsilon_i^\ell\epsilon_j^\ell$ which turns $\SS$
and $\SS^+$ into an abelian group. 
The resulting index $k$ may be denoted by $k=i*j$.
Notice that $j=0*j=j*0$ and $j*j=0$.
When $\varepsilon_i*\varepsilon_j=\varepsilon_k$ we will also
say that $\varepsilon_k$ is the \emph{translate} of
$\varepsilon_i$ by  $\varepsilon_j$.
Our presentation is pervaded with \emph{conjugate pairs}, that is pairs
$(i,j)$ of integers with $i+j=2^{n-1}-1=i*j$ and $i,j\ge0$.
\\
In Table \ref{Proof scheme newest} 
considerable effort 
has been undertaken to go from one row to another
by using the above product structure. The horizontal lines indicate a discontinuity
in this effort.
The resulting proof scheme has the nice property that it reduces to a proof scheme for
$n=8$ by considering only the even numbered rows, starting with 1, 3, \dots, 
and integer division of the
numbers in the table by 2.
This procedures works iteratively until $n=5$, ending with the last row of the table, with numbers
integer divided by 16.
We do not present a proof based on this scheme, since it is more elaborate then the proof based on
Table \ref{Proof scheme new experimental}.

\noindent
\begin{table}[hbt]
\setlength{\tabcolsep}{2.5pt}
\begin{tabular}{ccccccccccccccccl}
(0&255&(93&162)&(105&150)&(63&192))&(69&186)&(41&214)&((120&135)&24&231)\\
(1&254&(92&163)&(104&151)&(62&193))&(68&187)&(40&215)&(121&134)&(25&230)\\
\cline{5-6} 
(2&253&(95&160)&(106&149)&(61&194))&(71&184)&(43&212)&(122&133)&(26&229)\\
(3&252&(94&161)&(107&148)&(60&195))&(70&185)&(42&213)&(123&132)&(27&228)\\
\cline{3-4} 
(4&251&(91&164)&(108&147))&((59&196)&65&190)&((45&210)&124&131)&(28&227)\\

(5&250&(90&165))&(109&146)&((58&197)&64&191)&((44&211)&125&130)&(29&226)\\

(6&249&(89&166))&(110&145)&((57&198)&67&188)&((47&208)&126&129)&(30&225)\\
(7&248&(88&167))&(111&144)&((56&199)&66&189)&((46&209)&127&128)&(31&224)\\
\cline{5-6}\cline{9-16} 
(8&247&(85&170)&(102&153))&((55&200)&(75&180)&32&223)&((116&139)&20&235)\\
(9&246&(84&171))&(103&152)&((54&201)&(74&181)&33&222)&(117&138)&(21&234)\\
(10&245&(87&168))&(100&155)&((53&202)&(73&182)$^-$&34&221)&(118&137)&(22&233)\\
(11&244&(86&169))&(101&154)&((52&203)&(72&183)$^-$&35&220)&(119&136)&(23&232)\\
\cline{13-14} 
(12&243&(81&174))&(98&157)&(51&204)&((79&176)&(36&219)$^-$&(114&141)&16&239)\\
(13&242&(80&175))&(99&156)&(50&205)&((78&177)&(37&218)$^-$&(115&140)&17&238)\\
(14&241&(83&172))&(96&159)&(49&206)&(77&178)&(38&217)&((112&143)&18&237)\\
(15&240&(82&173))&(97&158)&(48&207)&(76&179)&(39&216)&((113&142)&19&236)\\

\end{tabular}
\caption{Proof scheme for $n=9$, numbers should not appear twice.
Terms enclosed in matching parentheses correspond to
twins, 4-tuplets, 6-tuplets, 8-tuplets.
A twin decorated with a $-$ sign is not needed in the $2k$-tuplet
in which it is contained. 
}\label{Proof scheme newest}

\end{table}

\subsection*{Lattice structure}
Consider the vectorspace of row vectors $\R^n$.
Given $R_1,R_2\in\R^n$ we will say that $R_1$ is less than $R_2$, $R_1\subseteq R_2$,
iff $R_1a\le R_2a$ for all $a\in\QQ$. 
Notice that
$$R_1\subseteq R_2
\Leftrightarrow\forall a\in\QQ:R_1a\le R_2a
\Leftrightarrow R_2-R_1\in\QQ^*
\Leftrightarrow \overline R_2-\overline R_1\curlyge0.$$
Given $R_1,R_2\in\R^n$ there is a unique least upper bound $R\in\R^n$, denoted by $R_1\vee R_2$.
Namely, 
$R$ such that the $i$-th coordinate of $\overline R$ equals 
$\overline R^i=\max(\overline R_1^i, \overline R_2^i)$.
Clearly then, $R_1\subseteq R$ and $R_2\subseteq R$.
In the same way there is a unique greatest lower bound, 
denoted by $R_1\wedge R_2$, namely $R$ such that
the $i$-th coordinate of $\overline R$ equals 
$\overline R^i=\min(\overline R_1^i, \overline R_2^i)$.
Clearly then, $R\subseteq R_1$ and $R\subseteq R_2$.
Notice also, that $R_1\wedge R_2=-((-R_1)\vee(-R_2))$.

\begin{lemma}
If $\varepsilon_1,\varepsilon_2\in\SS$, then $\varepsilon_1\vee\varepsilon_2\in\SS$
and $\varepsilon_1\wedge\varepsilon_2\in\SS$.
More generally, let $A$ be an arithmetic progression,
$A=\{p+q\cdot r\mid r=1,\ldots,k\}$ for some $p,q\in\R$, 
and $e_1,e_2\in A^n$, then $e_1\vee e_2, e_1\wedge e_2\in A^n$.
\end{lemma}
\begin{proof}
$(\varepsilon_1\vee\varepsilon_2)^1=\max(\overline\varepsilon_1^{1},\overline\varepsilon_2^{1})
=\max(\varepsilon_1^{1},\varepsilon_2^{1})$ and for $i>1$
we have
$(\varepsilon_1\vee\varepsilon_2)^i=\max(\overline\varepsilon_1^{i},\overline\varepsilon_2^{i})
-\max(\overline\varepsilon_1^{i-1},\overline\varepsilon_2^{i-1})
$ is odd integer, at least $-1$ and at most $+1$.
Thus $\varepsilon_1\vee\varepsilon_2\in\SS$.
\\
More generally, let $e_1,e_2\in A^n$.
Without loss of generality in the first place 
we may suppose that $q\ne 0$ and then we may suppose
$p=0$ and $q=1$, by considering the transformation $A\cong \{1,\ldots,k\}$ mapping $p+qr$ to $r$. 
Notice that
if for some $k=1,2$, $\max(\overline e_1^i,\overline e_2^i)=\overline e_k^i$
and $\max(\overline e_1^{i-1},\overline e_2^{i-1})=\overline e_k^{i-1}$
then $\max(\overline e_1^i,\overline e_2^i)-\max(\overline e_1^{i-1},\overline e_2^{i-1})=e_k^i\in A$.
Else, suppose without loss of generality
$\max(\overline e_1^i,\overline e_2^i)=\overline e_1^i$
and
$\max(\overline e_1^{i-1},\overline e_2^{i-1})=\overline e_2^{i-1}$.
Then $\overline e_1^i-\overline e_2^{i-1}$ is an integer with
$$
e_2^i
\le
\overline e_1^i-\overline e_2^i+e_2^i
=
\overline e_1^i-\overline e_2^{i-1}
=
\overline e_1^{i-1}-\overline e_2^{i-1}+e_1^i\le e_1^i
$$
so that $(e_1\vee e_2)^i=\overline e_1^i-\overline e_2^{i-1}
\in \{1,\ldots,k\}$.
\end{proof}
\\
\\
In particular, for $n=9$, the greatest lower bound 
of $\{\varepsilon_i\in\SS^+\mid R_3+\varepsilon_i\in\QQ^*\}$ is 
$\varepsilon_{219}$.
The greatest lower bound of 
$\{\varepsilon_i\in\SS^+\mid R_2+\varepsilon_i\in\QQ^*\}$ is $\varepsilon_{234}$.
The greatest lower bound of 
$\{\varepsilon_i\in\SS^+\mid R_3+\varepsilon_i\in\QQ^*\hbox{ and }
R_2+\varepsilon_i\in\QQ^*\}$ is the least upper bound of
$\{\varepsilon_{219},\varepsilon_{234}\}$, and that is $\varepsilon_{218}$.
The least upper bound of 
$\{\varepsilon_i\mid 123<i<128, \hbox{ or }187<i<192,
\hbox{ or }219<i<224, \hbox{ or } 236\le i\le 239\}$
is $\varepsilon_{124}$.
The least upper bound of 
$\{\varepsilon_i\mid i=231\hbox{ or }i=235\}$ is $\varepsilon_{231}$.
The least upper bound of 
$\{\varepsilon_i\mid 240\le i\le 255\}$ is $\varepsilon_{240}$.
The greatest lower bound of 
$\{\varepsilon_i\in\SS^+\mid \varepsilon_i\in\QQ^*\}$ is 
$\varepsilon_{170}$.

\subsection*{Programming}

In this subsection we expound the context in which we developed the proof of 
Theorem \ref{Main}.
\\Be given $e_1,\ldots, e_k\in\SS$.
We consider the problem of finding $(R,\lambda)\in\R^n\times\R_+^k$ such that $\frac12RR'$ is minimal, subject to 
$\lambda^1+\cdots+\lambda^k=1$ and $R-L\in\QQ^*$, where $L=\sum\lambda^ie_i$.
Let $Q$ be the $n\times n$-matrix with diagonal $-1$ and superdiagonal $+1$, so that $x\in\QQ$
is equivalent to $Qx\curlyle0$. Then $R-L\in\QQ^*$ is equivalent to $(R-L)Q^{-1}\curlyle0$.
Let $E$ be the $k\times n$-matrix with rows $e_1,\ldots,e_k$ and $\underline1$ the $k\times1$ matrix with coefficients $1$. 
Then $L=\lambda E$ and the conditions are:
\begin{align*}
(R-\lambda E)Q^{-1}&\curlyle0\\
-\lambda&\curlyle0\\\lambda\,\underline1&=1.
\end{align*}
In matrix form, where the big matrix is of dimension $(n+k)\times(n+k)$, one obtains
$$(R,\lambda)
\left(\begin{matrix}
Q^{-1}&0\\
-EQ^{-1}&-I
\end{matrix}\right)
\curlyle (0,0)
\hbox{ and } \lambda\,\underline1=1,
$$
where $I$ denotes the $k\times k$ identity matrix.
The dual program involves parameters $(u,v)\in\R_+^n\times\R_+^k$
and $w\in\R$.
The Lagrange dual goal function is
\begin{align*}
g(u,v,w)
&=
\min \{\frac12 a  a'
+
\left[(a,\ell)
\left(\begin{matrix}
Q^{-1}&0&0\\
-EQ^{-1}&-I&\underline1
\end{matrix}\right)
-(0,0,1)\right]
\left(\begin{matrix}
u\\
v\\
w
\end{matrix}\right)
\mid (a,\ell)\in\R^n\times\R^k\}
\\&=
\min\{-\frac12 a a'-w
\mid
(a,\ell) \hbox{ s.t. }
\left(\begin{matrix}a'\\0\end{matrix}\right)
+
\left(\begin{matrix}
Q^{-1}&0&0\\
-EQ^{-1}&-I&\underline1
\end{matrix}\right)
\left(\begin{matrix}
u\\
v\\
w
\end{matrix}\right)
=0
\}
\\&=
\begin{cases}
-\frac12 (Q^{-1}u)'(Q^{-1}u)-w,
\hbox{ if } -EQ^{-1}u-v+\underline1\,w=0;
\\
-\infty, \hbox{ else.}
\end{cases}
\end{align*}
For the optimal solution the parameters $(R,\lambda)$ and $(u,v,w)$ 
have to satisfy additionally 
the complementary slackness conditions 
$(R-\lambda E)Q^{-1}u=0$, $\lambda v=0$, 
and the crucial 
$$
0=
\left(\begin{matrix} R'\\0\end{matrix}\right)+
\left(\begin{matrix}
Q^{-1}&0&0\\
-EQ^{-1}&-I&\underline1
\end{matrix}\right)
\left(\begin{matrix} u\\v\\w\end{matrix}\right)
=
\left(\begin{matrix} R'+Q^{-1}u\\-EQ^{-1}u-v+\underline1\,w\end{matrix}\right)
\in\R^{n+k}
$$
to have $\max\{g(u,v,w)\mid (u,v)\curlyge0,w\}=\min\{\frac12RR'\mid (R,\lambda)\hbox{ such that }
\lambda\curlyge0, \lambda\,\underline1=1,\hbox{ and }R-\lambda E\in\QQ^*\}$.
In particular $QR'=-u\curlyle0$, so $R'\in\QQ$. 
Moreover $RR'=-RQ^{-1}u=-\lambda E Q^{-1}u=\lambda v-\lambda\,\underline1\,w=-w$,
and $v=ER'-\underline1\,RR'\curlyge0$.
\\
We finish with the following summary:
\begin{theorem}\label{SQP}
Be given $e_1,\ldots, e_k\in\SS$.
Consider the quadratic program of finding the minimum $(R,\lambda)\in\R^n\times\R^k$ of 
$\frac12RR'$, subject to $0\curlyle\lambda$,
$\lambda^1+\cdots+\lambda^k=1$ and $R-L\in\QQ^*$, where $L=\sum\lambda^ie_i$.
Then a feasible solution $(R,\lambda)$
is optimal
if and only if 
$R'\in\QQ$, and $e_iR'\ge RR'$ for $i=1,\dots,k$. 
\end{theorem}
\begin{proof}By inspection of the dual program. 
Notice that the theorem implies the complementary slackness conditions above,
that evaluate to $LR'=RR'$.
\end{proof}

\begin{remark}
Suppose $R$ and $\lambda$ are as in Theorem \ref{SQP}, and $(S,\mu)\in\R^n\times\R_+^k$ 
 is such that $\mu^1+\cdots+\mu^k=1$ and that,
with $M=\sum\mu^ie_i$ we have $S-M\in\QQ^*$.
Then $SS'-RR'=(S-R)(S-R)'+2(S-R)R'\ge(S-R)(S-R)'+2(M-R)R'= (S-R)(S-R)'$.
In particular there is exactly one $R$ for which there exists $\lambda$ as in Theorem \ref{SQP}.
\qed
\end{remark}
\begin{remark}\label{Remark5}
Be given $e_1,\ldots, e_k\in\SS$.
Suppose $(R,\lambda)\in\R^n\times\R_+^k$, 
is such that $\lambda^1+\cdots+\lambda^k=1$
and $R-L\in\QQ^*$, where $L=\sum\lambda^ie_i$, and suppose $RR'\le1$.
Then the following holds:
If $a\in\QQ$ is such that $e_ia\ge0$, $i=1,\ldots,k$, then
$0\le\min_i e_ia\le La\le Ra\le\|R\|\|a\|\le\|a\|$.
\qed
\end{remark}
The last remark coincides with Lemma \ref{lemma 2}.
Notice that given $R$ the search for a convex combination $L$ of $e_1,\ldots,e_k$ such that
$R-L\in\QQ^*$ corresponds to a linear optimization problem, finding maximal $x\in\R$ for which
$\overline R-\overline L\curlyge x$ which is successful if one finds $x\ge0$.
Conversely, given $L$, 
finding the solution of the minimization problem of $\frac12RR'$ subject to 
$R'\in\QQ$ and $R-L\in\QQ^*$ can be reduced to a linear programming problem.
\\
We made extensive use of 
quadratic programming software
to construct the proof schemes in Tables \ref{Proof scheme new experimental} and \ref{Proof scheme newest}.
Finally, R-package lpSolve\renewcommand{\thefootnote}{$\dagger$}\footnote{
Michel Berkelaar and others (2015). lpSolve: Interface to 'Lp\_solve'
  v. 5.5 to Solve Linear/Integer Programs. R package version 5.6.13.
}
is used 
to handle the cases in Table \ref{Proof scheme new experimental}, 
based on the library of vectors $R$ given in Table \ref{library} at the end of Section \ref{Proofsnew}. 
\begin{remark}\label{Remark6}
Be given $e_1,\ldots, e_k\in\SS$.
Suppose $R\in\R^n$ is such that $R'\in\QQ$, $e_iR'\ge RR'$ for $i=1,\dots,k$, and $RR'>1$.
Then the following holds:
$\min_i |e_iR'|=\min_i e_iR'\ge RR'>\|R\|$.
\qed
\end{remark}
We present the following example to Remark \ref{Remark5}: Consider $e_1=-\varepsilon_{223}=\varepsilon_{288}$
and $e_2=\varepsilon_{106}$.
Let $a\in\QQ$. 
Suppose that $\varepsilon_{223}a\le0$ and $\varepsilon_{106}a\ge0$.
Then we will prove that $\min(-\varepsilon_{223}a,\varepsilon_{106}a)\le\|a\|$.
Take $\lambda=(2,3)/5$ and $R= (3,3,1,1,1,1,1,1,1)/5$, then
\begin{align*}
L&=\frac25(-\varepsilon_{223})+\frac35\varepsilon_{106}
\\&=
\frac25(-1,1,1,-1,1,1,1,1,1)+\frac35(1,1,-1,-1,1,-1,1,-1,1)
\\&=
(1,5,-1,-5,5,-1,5,-1,5)/5,
\\
\overline L&=(1,6,5,0,5,4,9,8,13)/5,
\\
\overline R&=(3,6,7,8,9,10,11,12,13)/5.
\end{align*}
So that it is seen that $\overline L\curlyle\overline R$.
Moreover $\|R\|=1$ so that
$$\min(-\varepsilon_{223}a,\varepsilon_{106}a)
\le La\le Ra\le\|R\|\,\|a\|\le\|a\|.$$
Obviously, for $e_1^*=-\varepsilon_{222}$ and $e_2^*=\varepsilon_{107}$, given the same $\lambda$
we have
$$L^*=\frac25(-1,1,1,-1,1,1,1,1,-1)+\frac35(1,1,-1,-1,1,-1,1,-1,-1)\curlyle L,$$
so that obviously $\overline {L^*}\curlyle\overline L\curlyle R$.

Harrie Hendriks, Martien van Zuijlen

\emph{Radboud University Nijmegen}

\emph{Faculty of Science, Department of Mathematics}

\emph{Nijmegen, The Netherlands}

\texttt{\{H.Hendriks, M.vanZuijlen\}@science.ru.nl}
\end{document}